\documentclass[a4paper,12pt,reqno]{amsart}
\usepackage{amscd}
\usepackage{amsmath}
\usepackage{ulem}
\usepackage{amssymb}
\usepackage{amsthm}
\usepackage[usenames,dvipsnames]{color}
\usepackage{color,soul}
\usepackage{enumerate, mdwlist}  
\usepackage{tikz}
\usepackage{url}

\input xy
\xyoption{all}

\begin{document}

\newtheorem{thm}{Theorem}
\newtheorem{prop}[thm]{Proposition}
\newtheorem{conj}[thm]{Conjecture}
\newtheorem{lem}[thm]{Lemma}
\newtheorem{cor}[thm]{Corollary}
\newtheorem{axiom}[thm]{Axiom}
\newtheorem{sheep}[thm]{Corollary}
\newtheorem{deff}[thm]{Definition}
\newtheorem{fact}[thm]{Fact}
\newtheorem{example}[thm]{Example}
\newtheorem{slogan}[thm]{Slogan}
\newtheorem{remark}[thm]{Remark}
\newtheorem{quest}[thm]{Question}
\newtheorem{zample}[thm]{Example}

\newcommand{\sthat}{\hspace{.1cm}| \hspace{.1cm}}
\newcommand{\id}{\operatorname{id} }
\newcommand{\acl}{\operatorname{acl}}
\newcommand{\dcl}{\operatorname{dcl}}
\newcommand{\irr}{\operatorname{irr}}
\newcommand{\aut}{\operatorname{Aut}}
\newcommand{\fix}{\operatorname{Fix}}

\newcommand{\oo}{\mathcal{O}}
\newcommand{\aaa}{\mathcal{A}}
\newcommand{\mm}{\mathcal{M}}
\newcommand{\curg}{\mathcal{G}}
\newcommand{\bbf}{\mathbb{F}}
\newcommand{\A}{\mathbb{A}}
\newcommand{\R}{\mathbb{R}}
\newcommand{\Q}{\mathbb{Q}}
\newcommand{\C}{\mathbb{C}}
\newcommand{\cc}{\mathcal{C}}
\newcommand{\dd}{\mathcal{D}}
\newcommand{\N}{\mathbb{N}}
\newcommand{\Z}{\mathbb{Z}}
\newcommand{\cF}{\mathcal F}
\newcommand{\cB}{\mathcal B}
\newcommand{\cU}{\mathcal U}
\newcommand{\cV}{\mathcal V}
\newcommand{\cG}{\mathcal G}
\newcommand{\cD}{\mathcal D}
\newcommand{\curly}{\mathcal{C}}
\newcommand{\durly}{\mathcal{D}}
\newcommand{\fff}{\mathcal{F}}
\newcommand{\calc}{\mathcal{C}}
\newcommand{\GG}{\mathbb{G}}
\newcommand{\PP}{\mathbb{P}}
\newcommand{\Gal}{\mathrm{Gal}}
\newcommand{\Aut}{\mathrm{Aut}}
\newcommand{\signature}{\mathrm{sign}}

\definecolor{mypink3}{cmyk}{0, 0.7808, 0.4429, 0.1412}

\newcommand{\mahrad}[1]{{\color{blue} \sf $\clubsuit\clubsuit\clubsuit$ Mahrad: [#1]}}
\newcommand{\ramin}[1]{{\color{red}\sf $\clubsuit\clubsuit\clubsuit$ Ramin: [#1]}}
\newcommand{\ramintak}[1]{{\color{mypink3}\sf $\clubsuit\clubsuit\clubsuit$ Ramin2: [#1]}}
\newcommand{\ramintakloo}[1]{{\color{green} \sf $\clubsuit\clubsuit\clubsuit$ Ramin3 [#1]}}

\DeclareRobustCommand{\hlgreen}[1]{{\sethlcolor{green}\hl{#1}}}
\DeclareRobustCommand{\hlcyan}[1]{{\sethlcolor{cyan}\hl{#1}}}
\newcommand{\Fmodtor}{F^\times / \mu(F)}

\title{Counting Ideals in $\Z[t]/(f)$}
\author{Sarthak Chimni} 
\begin{abstract}
In this paper  we study the growth of ideals in $\Z[t]/(f)$ for a monic cubic polynomial $f$. We also compute the ideal zeta function of $\Z[t]/(t^n)$ for any $n \in \N$.
\end{abstract}
\maketitle
\section{Introduction}
Given a commutative ring R with identity whose additive group is isomorphic to $\Z^n$ for some $n \in \mathbb{N}$ we define
\[
a_R^{\iota}(k) =  |\{I \text{ ideal in } R \mid [R:S] = k\}|
\]

We define the ideal zeta function of $R$ to be
\[
\zeta_R^I(s) = \sum_{k=0}^{\infty} \frac{a_R^{\iota}(k)}{k^s}.
\]
Then we prove the following theorem:
\begin{thm} \label{main result}
For any $n \in \N$
\[
\zeta_{\Z[t]/(t^n)}^I(s) = \zeta(s) \zeta(2s-1) \zeta(3s-2) \cdots \zeta(ns - (n-1)).
\]

\end{thm}

(It has been brought to my notice by Michael Schein that Theorem \ref{main result} is Corollary 4.3 from \cite{Rossmann} and can also be proven by counting Hermite Normal Forms of matrices. The proof is left in detail as it is more elementary than that of  \cite{Rossmann}.) \\
It is easy to see that $\zeta_{\Z[t]/(t^n)}^I(s)$ has a pole at $s=1$ of order $n$. An application of a standard Tauberian Theorem then gives the following result.
\begin{cor} \label{growth}
Let c = $\frac{1}{n!(n-1)!}$ then
\begin{equation}
\sum_{k \leq B} a_{\Z[t]/(t^n)}^{\iota}(k) \sim cB (\log B)^{n-1}.
\end{equation}
\end{cor}

More generally for any monic polynomial $f \in \Z[t]$, denote $\Z[t]/(f)$ by $\Z_f$. Then the authors in \cite{FKS} also conjecture that if $f = g_1^{m_1}g_2^{m_2} \cdots g_k^{m_k}$ where $g_i$ is irreducible over $\Z[t]$ then $\zeta_{\Z_f}^I (s)$ has a pole at $s=1$ of order $\sum_{i=1}^{k} m_i$. We prove this for the case of cubic polynomials in the following theorem:

\begin{thm} \label{main 2}
Let $f$ be a monic cubic polynomial in $\Z[t]$, then $\zeta_{\Z_f}^I(s)$ converges for $\mathcal{R}(s) > 1$ and has a pole at $s=1$. Let $m_f$ denote the order of the pole of $\zeta_{\Z_f}^I$. Then $m_f$ is equal to the number of irreducible factors counted with multiplicity of $f$ in $\Z[t]$.
\end{thm} 

In fact the authors of \cite{FKS} believe that this holds for polynomials of any degree $n$. The ideal zeta function for $\Z[t]$ is computed by Segal in \cite{Segal} where he proves the following theorem in which the equality is to be interpreted as an identity of formal Dirichlet Series. 

\begin{thm}
Let $S$ be a Dedekind domain, not a field, having only finitely many ideals of each finite index. Let $R = S[t]$. Then R has only finitely many ideals of each finite index, and 

\begin{equation}
\zeta_R^I(s) = \prod_{j=1}^{\infty} \zeta_S^I(js-j)
\end{equation}
\end{thm}

\

The general theory to study these zeta functions using p-adic integration techiques was introduced by Grunewald, Segal and Smith in \cite{GSS}. They show that $\zeta_R(s)$ can be expressed as an Euler product of rational functions of $p^{-s}$ over all primes $p$. In \cite{KMT} Kaplan, Marcinek and Takloo-Bighash study subring growth of $\Z^n$ and more generally the distribution of orders in number fields by locating the rightmost poles of these zeta functions instead of computing them explicitly. I use these ideas in the proofs in this paper.

\

The paper is organised is as follows. In Section \ref{p-adic} we introduce the p-adic integration techniques we use in our proofs. We prove Theorem \ref{main result} in Section \ref{section 3} and Theorem \ref{main 2} in Section \ref{cubic}.

\section{P-adic setting} \label{p-adic}
Let $R$ be a commutative ring with identity whose additive group is isomorphic to $\Z^n$ for some $n \in \mathbb{N}$. Then the following theorem is a summary of results from \cite{GSS}: 
\begin{thm} \label{GSS}
1. \space The series $\zeta_R^I(s)$ converges in some right half plane of $\mathbb{C}$. The abscissa of convergence $\alpha_R^I$ of $\zeta_R^<(s)$ is a rational number. There is a $\delta > 0$ such that $\zeta_R^I(s)$ can be meromorphically continued to the domain $\{s \in \mathbb{C} \mid \mathcal{R}(s) > \alpha_R^I - \delta\}$. Furthermore, the line $\mathcal{R}(s) = \alpha_R^I$ contains at most one pole of $\zeta_R^<Is)$ at the point $ s = \alpha_R^I$.\\\\
2. \space There is an Euler product decomposition 

\begin{equation}
\zeta_R^I(s) = \prod_p \zeta_{R,p}^I(s)
\end{equation}
with the local Euler factor given by 
\begin{equation*}
\zeta_{R,p}^I(s) = \sum_{l=0}^{\infty} \frac{a_R^{\iota}(p^l)}{p^{ls}}.
\end{equation*}
This local factor is a rational function of $p^{-s}$; there are polynomials $P_p, Q_p \in \Z[x]$ such that $\zeta_{R,p}^I(s) = P_p(p^{-s})/Q_p(p^{-s})$. The polynomials $P_p, Q_p$ can be chosen to have bounded defree as p varies. 
\end{thm}

By a theorem of Voll \cite{Voll}, the local Euler factors satisfy functional equations. The paper \cite{GSS} introduced a $p$-adic formalism to study the local Euler factors $\zeta_{R,p}^<(s)$. Fix a $\Z$-basis for $R$ and identify $R $ with $\Z^n$. The multiplication in $R$ is given by a bi-additive map
\[
\beta : \Z^n \times \Z^n \to \Z^n
\]
which extends to a bi-additive map
\[
\beta_p : \Z_p^n \times \Z_p^n \to \Z_p^n
\]
giving $R_p = R \otimes_{\Z} \Z_p$ the structure of a $\Z_p$ -algebra. 

\

Let $\mathcal{M}_p(\beta)$ be the subset of the set of $n \times n$ lower triangular matrices $M$ with entries in $\Z_p$ such that if the rows of $M = (x_{ij})$ are denoted by $v_1, \dots v_n$,  and then for all $j$ satisfying $1 \leq j \leq n$ and for all $v\in \Z_p^n$, there are $p$-adic integers $c_{1}^v, \dots c_{n}^v$
\begin{equation} \label{multiplicativity}
\beta_p(v,v_j) = \sum_{k=1}^n c_{k}^vv_k.
\end{equation}
Let $dM$ be the normalized additive Haar measure on $T_n(\Z_p)$, the set of $ n \times n$ lower triangular matrices with entries in $\Z_p$. Proposition 3.1 of \cite{GSS} says :
\begin{equation} \label{p-adic integral}
\zeta^I_{R,p}(s) = (1-p^{-1})^{-n} \int_{\mathcal{M}_p(\beta)} |x_{11}|^{s-n+1} |x_{22}|^{s-n+2}\cdots |x_{n,n}|^{s}dM.
\end{equation}

\

We now apply these considerations to the case of $\Z_f = \Z[t]/(f)$ where $f$ is a polynomial of degree $n$ in $\Z[t]$. We fix $B = \{t^{n-1}, \dots, t,1\}$ as an ordered basis for $\Z_f$ as a lattice. So that $t^{n-j}$ corresponds to $e_j$ where $e_j$ denotes the $j^{\text{th}}$ standard basis vector. Then the product of two vectors $v \cdot w$ is the vector representing the product of the corresponding polynomials in $Z_f$ in basis $B$. By Theorem \ref{GSS} there exists an Euler product decomposition  
\[
\zeta_{\Z_f}^I(s) = \sum_{m=1}^{\infty} \frac{a_{\Z_f}^{\iota}(m)}{m^s} =  \prod_{p \text{ prime}} \zeta_p^I(\Z_f,s)
\]
where 
\[
\zeta_p^I(\Z_f,s) = \sum_{k=1}^{\infty} \frac{a_{k,p}^{\iota}}{p^{ks}}.
\]
Here $a_{k,p}$ counts the number of ideals of $\Z_f$ of index $p^k$.  

\

Let $v_f$ be the vector corresponding to $f$ in basis $B$. We define $\mathcal{M}_f(p)$ be the subset of the set of $n \times n$ lower triangular matrices $M$ with entries in $\Z_p$ such that if the rows of $M = (a_{ij})$ are denoted by $v_1, \dots v_n$, then for all $j$ satisfying $1 \leq j \leq n$, there are $p$-adic integers $\{c_{kj}\}_{k=1}^n$ such that 
\begin{equation} \label{ideal condition}
v_t \cdot v_j = \sum_{k=1}^n c_{kj}v_k
\end{equation}
Observe that $M \in \mathcal{M}_f(p)$ if and only if the rows of $M$ generate an ideal in $\Z_f$. Using Equation (\ref{p-adic integral}) we find that the local factors are given by
\[
\zeta_p^I(\Z_f,s) = (1-p^{-1})^{-n} \int_{\mathcal{M}_f(p)} |a_{11}|^{s-n+1} |a_{22}|^{s-n+2}\cdots |a_{n,n}|^{s}dM.
\]
 In order to compute these local factors we need the following definition.
\begin{deff} \label{domain}
Let $(b_1, b_2, \dots, b_n)$ be a n-tuple of non negative integers, we set
\[
\mathcal{M}_f(p;b_1,\dots,b_n) = \Bigg\{ M = 
\begin{bmatrix}
p^{b_1}  &  0     &   \dots   &  0   &     \\
a_{21}   &   p^{b_2} & \dots & \vdots   &  \\
\vdots   &     & \ddots  &    0 \\
a_{n1}   &   \dots  &       &   p^{b_n}
\end{bmatrix}
\in \mathcal{M}_f(p) \Bigg\}.
\]
\end{deff}

Let $\mu_p^f(b_1, \dots, b_n)$ be the volume of $\mathcal{M}_f(p;b_1,\dots,b_n)$ as a subset of $\Z_p^{\frac{n(n+1)}{2}}$. It follows from equation (\ref{p-adic integral}) that 
\begin{equation} \label{local factor 1}
\zeta_p^I(\Z_f,s) = \sum_{b_i = 0}^{\infty} \frac{p^{(n-1)b_1 + (n-2)b_2 + ... + 2b_{n-2} + b_{n-1}}}{p^{(b_1+ \dots +b_n)s}} \mu_p^f(b_1, \dots, b_n).
\end{equation}

\section{Proof of Theorem 1} \label{section 3}

We now prove Theorem \ref{main result}. In this section $f(t) = t^n$ and  $\mathcal{M}_f(p;b_1,\dots,b_n)$ is denoted by $\mathcal{M}_n(p;b_1,\dots,b_n)$ and $\mu_p^f(b_1, \dots, b_n)$ is denoted by $\mu_p(b_1, \dots, b_n)$.

We use equation (\ref{ideal condition}) to compute $\mu_p(b_1, \dots, b_n)$. Define $G_M$ be the subgroup of $\Z_p^n$ generated by the rows of $M$. Then we use the following standard result to find $\mu_p(b_1, \dots, b_n)$.

\begin{lem} \label{determinant}
Let $M \in \mathcal{M}_n(p;b_1,\dots,b_n)$, then $\mu_p(G_M) = p^{-(b_1+ \cdots +b_n)}$.
\end{lem}

\begin{prop} \label{Volume}
\[
\mu_p(b_1,\dots,b_n) = 
\begin{cases}
p^{-((n-2)b_1 + (n-3)b_2 + \dots + b_{n-2})} & b_1 \leq \cdots \leq b_n   \\
0 & \text{otherwise} 
\end{cases}
\]
\end{prop}

\begin{proof}
For any matrix $M$, let $M_j$ denote the matrix obtained from $M$ by deleting the last $n-j$ rows. Now  $M = (v_1,v_2,\dots, v_n)^T \in \mathcal{M}_n(p)$  if and only if  for all $ 2 \leq j \leq n, v_t \cdot v_j \in G_{M_{j-1}}$. Therefore $\exists c_{ij} \in \Z_p$ such that 
\begin{equation} \label{ideal condition 2}
v_t \cdot v_j =  (a_{j2}, \dots, a_{j,j-1}, p^{b_j}, 0,\dots, 0) = \sum_{i = 1}^{j-1} c_{ij}v_i.
\end{equation}
\[
\]

This happens only if $b_{j-1} \leq b_j$ as $c_{j-1,j} = p^{b_j - b_{j-1}}$. For $j=2$ Equation \ref{ideal condition 2} is satisfied if and only if $b_1 \leq b_2$ and this holds on a volume of $1$. For $3 \leq j \leq n$ we can write $v_tv_j = u_j + w_j$ where $ u_j \in G_{M_{j-2}}$ and $w_j = p^{b_j}e_{j-1}$ where $\{e_1,\dots,e_n\}$ is the standard basis for $\Z_p^n$. So equation (\ref{ideal condition 2}) holds on a volume of $\mu_p(G_{M_{j-2}})$. Using Lemma \ref{determinant} we have $\mu_p(G_{M_{j-2}}) = p^{-(b_1+ \cdots + b_{j-2})}$  for $3 \leq j \leq n$. This implies 

\[
\mu_p(b_1, \dots, b_n) = \prod_{j=3}^{n} \mu(G_{M_{j-2}}) =  \mu_p(b_1, \dots, b_n) = p^{-((n-2)b_1 + (n-3)b_2 + \dots + b_{n-2})}.
\]

\end{proof}

We now compute the local factors $\zeta_{R,p}(s)$. In order to this we need another lemma.

\begin{lem}\label{summation lemma}
For $2 \leq k \leq n$
\[
 \sum_{b_k \geq b_{k-1}} (px)^{b_k} \cdots \sum_{b_{n-1} \geq b_{n-2}} (px)^{b_{n-1}} \sum_{b_n \geq b_{n-1}} x^{b_n}  = (p^{n-k}x^{n-k+1})^{b_{k-1}} \prod_{j=1}^{n-k+1} (1-p^{j-1}x^j)^{-1}.
\]
\end{lem}

\begin{proof} 
We prove this by induction on the number of summations. If $k = n$, that is there is only one summation then it is easily seen that the equality holds. Now assume the equation is true for $n-k$ summations. Then we have 
\[
\sum_{b_k \geq b_{k-1}} (px)^{b_k} \cdots \sum_{b_{n-1} \geq b_{n-2}} (px)^{b_{n-1}} \sum_{b_n \geq b_{n-1}} x^{b_n} = \prod_{j=1}^{n-k} (1-p^{j-1}x^j)^{-1} \sum_{b_k \geq b_{k-1}} (px)^{b_k} (p^{n-k-1}x^{n-k})^{b_{k}}.
\]
So that
\[
\sum_{b_k \geq b_{k-1}} (px)^{b_k} \cdots \sum_{b_{n-1} \geq b_{n-2}} (px)^{b_{n-1}} \sum_{b_n \geq b_{n-1}} x^{b_n} = (p^{n-k}x^{n-k+1})^{b_{k-1}} \prod_{j=1}^{n-k+1} (1-p^{j-1}x^j)^{-1}.
\]
\end{proof}

\begin{prop} 
Let $x = p^{-s}$ then 
\begin{equation} \label{local factor 2}
\zeta_p^I(\Z[t]/(t^n),s)  = \prod_{j=1}^n (1-p^{j-1}x^j)^{-1}.
\end{equation}
\end{prop}

\begin{proof}

Applying Proposition \ref{Volume} to equation (\ref{local factor 1}) we have
\[
\zeta_p^I(\Z[t]/(t^n),s)  = \sum_{0 \leq b_1 \leq \cdots \leq b_n } \frac{p^{b_1 + \dots + b_{n-1}}}{p^{(b_1+ \dots +b_{n})s}}.
\]
Set $x = p^{-s}$. Then

\begin{equation} \label{summation}
\zeta_p^I(\Z[t]/(t^n),s)  = \sum_{b_1 \geq 0} (px)^{b_1} \sum_{b_2 \geq b_1} (px)^{b_2} \cdots \sum_{b_{n-1} \geq b_{n-2}}.(px)^{b_{n-1}} \sum_{b_n \geq b_{n-1}} x^{b_n}.  
\end{equation}

We apply Lemma \ref{summation lemma} with $k = 2$ to equation (\ref{summation}) to obtain
\[
\zeta_p^I(\Z[t]/(t^n),s)  =  \prod_{j=1}^{n-1} (1-p^{j-1}x^j)^{-1}    \sum_{b_1 \geq 0} (p^{n-1}x^{n})^{b_{1}}.
\]
Therefore
\[
\zeta_p^I(\Z[t]/(t^n),s)  = \prod_{j=1}^{n} (1-p^{j-1}x^j)^{-1}.
\]

\end{proof}

Multiplying the local factors given in equation (\ref{local factor 2}) gives the result.
\[
\zeta_{\Z[t]/(t^n)}^I(s) = \zeta(s) \zeta(2s-1) \zeta(3s-2) \cdots \zeta(ns - (n-1)).
\]

\section{ Ideals in cubic rings} \label{cubic}
In this section we prove Theorem \ref{main 2}. Before we prove the theorem we first state the following general result from \cite{FKS}:

\begin{thm} \label{FKS main} 
Let $f \in \Z[t]$ be monic and assume that $f = g_1 \cdots g_k$ with $g_1, \dots , g_k$
in $\Z[t]$ irreducible, monic and pairwise distinct. Then the ideal zeta function $\zeta_{\Z_f} (s)$
converges for $\mathcal{R}(s) > 1$, has a meromorphic extension to the halfplane $\{s \in \mathbb{C} \mid \mathcal{R}(s) > 0\}$
and has a pole of order $k$ at $s = 1$.  In particular, $\sum_{k \leq N} a_{\Z_f}(k) \sim cN(\log N)^{k-1}$
for some constant $c$. 
\end{thm}

We now prove Theorem \ref{main 2}.
\begin{proof}[Proof of Theorem \ref{main 2}]
In the case that the roots of $f$ are distinct the result follows from Theorem \ref{FKS main}. If $f = (t-\lambda)^3$ for some $\lambda \in \Z$ then the result follows from Theorem \ref{main result} by a change of variable. Hence we need to show that $m_f = 3$ when $f = (t-\lambda_1)^2(t-\lambda_2)$. Now note that since $f$ is not separable it is reducible and therefore $\lambda_1$ and $ \lambda_2 \in \Z$. Thus by a change of variable we can reduce this to the case $f = t^2(t-\lambda)$ with $\lambda \in \Z \setminus \{0\}$. The result now follows from Proposition \ref{lambda}.
\end{proof}

\begin{prop} \label{lambda}
Let $f(t) = t^2(t-\lambda) \in \Z[t]$. Then $\zeta_{\Z_f}^I(s)$ converges for $\mathcal{R}(s) > 1$ and has a pole of order $3$ at $s=1$.
\end{prop}

The rest of this section goes into proving Proposition \ref{lambda}. In what follows $f(t) = t^2(t-\lambda)$.
We find the local factors $\zeta_p(\Z_f,s)$ using the p-adic integration methods of Section 2 and 3 and use them to study the convergence of $\zeta_{\Z_f}(s)$.

Let $\mathcal{M}_f(p;b_1,b_2,b_3)$ and $\mu_p^f(b_1,b_2,b_3)$ be as in Definition $\ref{domain}$. Then an application of Equation (\ref{p-adic integral}) gives the following expression for the local factors of   $\zeta_{\Z_f}(s)$:

\begin{equation} \label{local factor 3}
\zeta_p^I(\Z_f,s) = \sum_{b_i= 0}^{\infty} \frac{p^{2b_1 +b_2}}{p^{b_1+b_2+b_3}s} \mu_p^f(b_1,b_2,b_3).
\end{equation}

In order to compute $\mu_p^f(b_1,b_2,b_3)$ we use the following lemma
\begin{lem}
$M = 
\begin{bmatrix}
p^{b_1}    &   0   &    0 \\
a_{21}    &   p^{b_2}   & 0 \\
a_{31}   &  a_{32}   &  p^{b_3}  \\
\end{bmatrix}
\in \mathcal{M}_f(p;b_1,b_2,b_3)$ if and only if the entries of $M$ satisfy the following inequalities.

\begin{equation} \label{ineq 1}
b_2 \leq v(p^{b_3}+ \lambda a_{31})
\end{equation}

\begin{equation} \label{ineq 2}
b_1+b_2 \leq v(p^{b_2} a_{32} - (p^{b_3} + \lambda a_{31})a_{21})
\end{equation}

\begin{equation} \label{ineq 3}
b_2 \leq v(a_{21})
\end{equation}

\begin{equation} \label{ineq 4}
b_1+b_2 \leq v(p^{2b_2} - \lambda a_{21}^2)
\end{equation}

\begin{equation} \label{ineq 5}
b_2 \leq b_1.
\end{equation}
\end{lem}

\begin{proof}
These inequalities follow from an application of  Equation (\ref{ideal condition}).
\end{proof}

\begin{lem}
If $(\lambda,p) = 1$ then 
\[
\mu_p^f(b_1,b_2,b_3) =p^{-b_1 - 2b_2 -\lceil \frac{b_1-b_2}{2} \rceil}
\]
\end{lem}

\begin{proof} \label{volume 1}
Suppose $(\lambda,p) = 1$.Using Inequality (\ref{ineq 3}) we can write $a_{21} = p^{b_2}z$ for some $ z \in \Z_p$. Therefore Inequality (\ref{ineq 4}) can be rewritten as
$b_1 - b_2 \leq v(1-\lambda z^2)$. Since Inequality (\ref{ineq 5}) holds, we have that these inequalities hold on a volume of $p^{-b_2 - \lceil \frac{b_1-b_2}{2} \rceil}$ for $a_{21}$. Now Inequality (\ref{ineq 1}) holds on a volume of $p^{-b_2}$ while Inequality (\ref{ineq 2}) holds on a volume of $p^{-b_1}$. Multiplying these volumes we get $\mu_p^f(b_1,b_2,b_3) = p^{-b_1 - 2b_2 -\lceil \frac{b_1-b_2}{2} \rceil}$ in this case.
\end{proof}
 
\
\begin{lem} \label{volume 2}
If $(\lambda,p) > 1$  then $\mu_p^f(b_1,b_2,b_3) \leq p^{-2b_2}$ and $b_1=b_2$.
\end{lem}

\begin{proof}
Now suppose $\lambda = p^a$ for some $a \in \mathbb{N}$. As before we can write $a_{21} = p^{b_2}z$ and therefore we have $b_1 - b_2 \leq v(1-p^a z^2)$. This inequality only holds if $b_1=b_2$. This implies that Inequality (\ref{ineq 2}) holds on a volume of $p^{-b_2}$. Therefore all the inequalities hold on a volume of atmost $p^{-2b_2}$.
\end{proof}

We now evaluate the local factors using Equation (\ref{local factor 3}).  

\begin{lem} \label{coprime}
Let $x = p^{-s}$. If $(\lambda,p) = 1$ then 
\[
\zeta_p^I(\Z_f,s) = \frac{1-x^2 + p^{-1}x  - p^{-1}x^2+ x^2-x^3}{(1-x)^2(1-px^2)(1-x^4)}
\]
\end{lem}

\begin{proof}
\[
\zeta_p^I(\Z_f,s) = \sum_{b_2 \leq b_1} \sum_{b_3= 0}^{\infty} \frac{p^{2b_1 +b_2}}{p^{b_1+b_2+b_3}s} \mu_p^f(b_1,b_2,b_3)
\]

Using Lemma \ref{volume 1}
\[
= \sum_{b_2 \leq b_1} \sum_{b_3= 0}^{\infty} \frac{p^{b_1 -b_2- \lceil \frac{b_1-b_2}{2} \rceil}}{p^{(b_1+b_2+b_3)s}} 
\]

\[
= A_{00} + A_{01} + A_{10} + A_{11}.
\]

Where 
\[
A_{ij} =  \sum_{k_2 \leq k_1} \sum_{b_3= 0}^{\infty} \frac{p^{2k_1 + i -(2k_2+j)- \lceil \frac{2k_1+i-(2k_2+j)}{2} \rceil}}{p^{(2k_1+i+2k_2+j+b_3)s}}
\]

$b_1 = 2k_1+i$ and $ b_2 = 2k_2 + j$.

Setting $x = p^{-s}$, $ b_1 = 2k_1$ and $b_2 = 2k_2$  we have
\[
A_{00} = \sum_{k_2 = 0}^{\infty} (p^{-1}x^2)^{k_2} \sum_{k_1 = k_2}^{\infty} (px^2)^{k_1} \sum_{b_3 = 0}^{\infty} (x)^{b_3}
\]

\[
= \frac{1}{1-x}  \sum_{k_2 = 0}^{\infty} (p^{-1}x^2)^{k_2} \sum_{k_1 = k_2}^{\infty} (px^2)^{k_1}
\]

\[
= \frac{1}{(1-x)(1-px^2)}\sum_{k_2 = 0}^{\infty} x^{4k_2}
\]

\[
A_{00}= \frac{1}{(1-x)(1-px^2)(1-x^4)}.
\]

We compute $A_{01}, A_{10}$ and $A_{11}$ similarly to get 

\[
\zeta_p(\Z_f,s) = \frac{1+x + p^{-1}x + x^2}{(1-x)(1-px^2)(1-x^4)}
\]

\[
= \frac{1-x^2 + p^{-1}x  - p^{-1}x^2+ x^2-x^3}{(1-x)^2(1-px^2)(1-x^4)}.
\]
\end{proof}

\begin{lem} \label{not coprime}
If $(\lambda,p) > 1$ then $\zeta_p(\Z_f,s)$ converges for $\mathcal{R}(s) > 1/2$.
\end{lem}

\begin{proof}
Let $\mathcal{R}(s) = \sigma$, an application of Lemma \ref{volume 2} gives 

\[
|\zeta_p^I(\Z_f,s)| \leq \sum_{b_2,b_3} \frac{p^{b_2}}{p^{(2b_2+b_3)\sigma}}
\]
Therefore 
\[
|\zeta_p^I(\Z_f,s)| \leq \frac{1}{(1-p^{-\sigma})(1-p^{1-2\sigma})}
\]

Since the right hand side converges for $\sigma > 1/2$ so does $\zeta_p(\Z_f,s)$.
\end{proof}

Now let $$F(s) = \prod_{p: (\lambda,p) > 1} \zeta_p(\Z_f,s)$$ and $$G(s) = \prod_{p: (\lambda,p) = 1} \zeta_p(\Z_f,s).$$

Then it is easy to see that $F(s)$ is a finite product and therefore converges for $\mathcal{R}(s) > 1/2$ by Lemma \ref{not coprime}.  On the other hand it follows from Lemma \ref{coprime} that $G(s)$ has a pole at $s=1$ of order $3$. Hence $\zeta_{\Z_f}(s) = F(s) G(s)$ converges for $\mathcal{R}(s) > 1$ and has a pole of order $3$ at $s=1$,  proving Proposition \ref{lambda}.

\end{document}